\documentclass[12pt]{amsart}
\usepackage{a4wide}
\usepackage[utf8]{inputenc}
\usepackage[T1]{fontenc} 
\usepackage[english]{babel}
\usepackage{amsmath}
\usepackage{amsfonts}
\usepackage{amssymb}
\usepackage{amsthm}
\usepackage{graphicx}
\usepackage{subcaption}
\usepackage{enumerate}
\usepackage{color}
\usepackage{todonotes}
\usepackage{bbm}
\usepackage{verbatim}

\definecolor{mygreen}{rgb}{0., 0.5, 0.}

\renewcommand{\d}{\,\mathrm{d}}
\newcommand{\E}[1]{\mathbb{E}\left[#1\right]} 
\newcommand{\no}[1]{\Vert#1\Vert} 
\newcommand{\nos}[1]{\Vert#1\Vert^2} 
\newcommand{\be}[1]{\left\vert#1\right\vert} 
\newcommand{\bes}[1]{\vert#1\vert^2} 
\renewcommand{\d}{\,\mathrm{d}}
\renewcommand{\E}[1]{\mathbb{E}\left[#1\right]} 
\renewcommand{\no}[1]{\Vert#1\Vert} 
\renewcommand{\nos}[1]{\Vert#1\Vert^2} 
\renewcommand{\be}[1]{\left\vert#1\right\vert} 
\renewcommand{\bes}[1]{\vert#1\vert^2} 

\newcommand{\fe}[1]{\frac{\nabla #1}{ \sqrt{ \vert \nabla #1 \vert^2 +\epsilon^2}}} 

\newcommand{\R}{\mathbb{R}}

\newcommand{\into}{\int\limits_{\mathcal{O}}}
\renewcommand{\epsilon}{\varepsilon}
\newcommand{\eps}{\epsilon}
\newcommand{\ska}[1]{\left( #1 \right)} 
\renewcommand{\div}{\mathrm{div}}

\newcommand{\intt}{\int_0^t}

\newcommand{\F}{\mathcal{F}}

\newcommand{\Hz}{\mathbb{H}^1_0}
\renewcommand{\L}{\mathbb{L}^2}

\renewcommand{\phi}{\varphi}
\renewcommand{\F}{\mathcal{F}}
\renewcommand{\P}{\mathbb{P}}

\renewcommand{\O}{\mathcal{O}}

\newcommand{\Xe}{X^{\epsilon}}

\renewcommand{\F}{\mathcal{F}}
\renewcommand{\P}{\mathbb{P}}

\renewcommand{\O}{\mathcal{O}}

\newtheorem{thms}{Theorem}[section]
\newtheorem{defs}{Definition}[section]
\newtheorem{cors}{Corollary}[section]

\newtheorem{bems}{Remark}[section]
\newtheorem{lemma}{Lemma}[section]

\begin{document}
\title[Correction to: Numerical approximation of the stochastic TV flow]{Correction to: Convergent numerical approximation of the stochastic total variation flow}

\author{\v{L}ubom\'{i}r Ba\v{n}as}
\address{Department of Mathematics, Bielefeld University, 33501 Bielefeld, Germany}
\email{banas@math.uni-bielefeld.de}
\author{Michael R\"ockner}
\address{Department of Mathematics, Bielefeld University, 33501 Bielefeld, Germany; Academy of Mathematics and Systems Science, CAS, Beijing }
\email{roeckner@math.uni-bielefeld.de}
\author{Andr\'e Wilke}
\address{Department of Mathematics, Bielefeld University, 33501 Bielefeld, Germany}
 \email{ awilke@math.uni-bielefeld.de}

\thanks{Funded by the Deutsche Forschungsgemeinschaft (DFG, German Research Foundation) - SFB 1283/2 2021 - 317210226.}

\begin{abstract}
We correct two errors in our paper \cite{our_paper}.
First error concerns the definition of the SVI solution,
where a boundary term which arises due to the Dirichlet boundary condition, was not included.
The second error concerns the discrete estimate \cite[Lemma 4.4]{our_paper}, which involves the discrete Laplace operator.
We provide an alternative proof of the estimate in spatial dimension $d=1$ by using a mass lumped version of the {discrete} Laplacian. 
Hence, after a minor modification of the fully discrete numerical scheme the convergence in $d=1$
follows along the lines of the original proof.
The convergence proof of the time semi-discrete scheme, which relies on the continuous counterpart of the estimate \cite[Lemma 4.4]{our_paper}, remains valid in higher spatial dimension. 
{The convergence of the fully discrete finite element scheme from \cite{our_paper} in any spatial dimension is shown in \cite{stvf_new} by using a different approach.}
\end{abstract}
\maketitle

\section{Introduction}
Let $\O \subset \R^d$ be an open convex domain with piecewise smooth boundary. We consider numerical approximation of the stochastic total variation flow
\begin{align}\label{TVF}
\d X&= \div\left(\frac{\nabla X}{\be{\nabla X}}\right) \d t -\lambda (X - g) \d t +X\d W, &&\text{in } (0,T)\times \O, \nonumber\\
X & = 0 && \text{on } (0,T)\times \partial \O, \\
X(0)&=x_0 &&\text{in } \O, \nonumber
\end{align}
which is constructed via the discretization of the regularized problem
 \begin{align}\label{reg.TVF}
 \d \Xe&=  \div\left(\frac{\nabla \Xe}{\sqrt{|\nabla \Xe |^2+\epsilon^2}}\right)\d t-\lambda(\Xe-g)\d t+\Xe\d W &&\text{in } (0,T)\times \O, \nonumber\\
 \Xe & = 0 && \text{on } (0,T)\times \partial \O, \\
 \Xe(0)&=x_0 &&\text{in } \O\,.\nonumber
\end{align}
Throughout the paper we employ the notation from \cite{our_paper}.
The first error is corrected in Section~\ref{sec_svi}
and the correction of the second error is provided in Section~\ref{sec_lap}.


\section{Definition of the SVI solution and the uniqueness proof}\label{sec_svi}

In the proof of \cite[Theorem 3.1]{our_paper} the term $IV$ in (29) is wrongly {rewritten } as
\begin{align*}
IV = -\ska{X_1^{\epsilon}-X_{2,n}^{\epsilon,\delta},\div \fe{X_{2,n}^{\epsilon,\delta}}}=\ska{\nabla X_1^{\epsilon}-\nabla X_{2,n}^{\epsilon,\delta}, \fe{X_{2,n}^{\epsilon,\delta}}}\,
\end{align*}
{since $X_1^{\epsilon}$ is only in $BV(\O)$  and  possibly non-zero at the boundary.}
Hence, to show the {uniqueness of the SVI solutions for $\eps>0$} requires a modification of the definition which takes into account the value at the solution at the boundary.
This definition  is consistent with the one from \cite{Roeckner_TVF_paper}, {which also shows uniqueness in case $\eps=0$}.

We define the following functionals which include the corresponding boundary terms
\begin{align*}
\bar{ \mathcal{J}}_{\eps,\lambda}(u)=
\begin{cases}
\mathcal{J}_{\eps,\lambda}(u) + \int_{\partial \O} \be{\gamma_0(u)} \d \mathcal{H}^{n-1} ~~ & \text{for}~ u \in BV(\O)\cap L^2(\O),\\
+\infty  ~~ &\text{for}~ u \in BV(\O)\setminus L^2(\O),
\end{cases}
\end{align*}
and 
\begin{align*}
\bar{ \mathcal{J}}_\lambda (u)=
\begin{cases}
\mathcal{J}_\lambda (u) + \int_{\partial \O} \be{\gamma_0(u)} \d \mathcal{H}^{n-1}  ~~ & \text{for}~ u \in BV(\O)\cap L^2(\O),\\
+\infty   ~~ & \text{for}~ u \in BV(\O)\setminus L^2(\O),
\end{cases}
\end{align*}
where $\gamma_0(u) $ is the trace of $u\in BV(\O)$ on the boundary and $\d \mathcal{H}^{n-1}$
is the Hausdorff measure on $\partial\O$.
$\bar{\mathcal{J}}_{\eps,\lambda}$ and $\bar{\mathcal{J}}_{\lambda}$ are  both convex and lower semicontinuous on $L^2(\O)$ and  the lower semicontinuous hulls of $\bar{\mathcal{J}}_{\eps,\lambda}\vert_{\Hz}$ or $\bar{\mathcal{J}}_{\lambda}\vert_{\Hz}$ respectively, cf. \cite[Proposition 11.3.2]{book_attouch}.
We define the SVI solution as follows.
\begin{defs}\label{def_varsoleps}
Let $0  < T < \infty$, {$\varepsilon \in [0,1]$} and {$x_0 \in L^2(\Omega,\F_0;\L)$ and $g \in \L$}.
Then a $(\F_t)$-{adapted} map {$\Xe \in L^2(\Omega; C([0,T];\L))\cap  L^1(\Omega; L^1((0,T);BV(\O)))$ 
(denoted by $X \in L^2(\Omega; C([0,T];\L))\cap  L^1(\Omega; L^1((0,T);BV(\O)))$ for $\eps=0$)}
is called an {SVI  solution} of (\ref{reg.TVF}) (or (\ref{TVF}) if $\varepsilon=0$) if $\Xe(0)=x_0$ ($X(0)=x_0$), and
for each $(\F_t)$-progressively measurable process $G\in L^2(\Omega \times (0,T),\L)  $ and for each $(\F_t)$-adapted $\L$-valued process 
{$Z$} with $\P$-a.s. continuous sample paths, {s.t. $Z \in L^2(\Omega \times (0,T);\Hz)$}, which satisfy the equation 
\begin{align*}
\d Z(t)= -G(t) \d t +Z(t)\d W(t), ~ t\in[0,T],
\end{align*}
it holds for {$\eps \in (0,1]$} that
\begin{align}\label{reg.SVI}
\frac{1}{2}& \E{\nos{\Xe(t)-Z(t)}}+\E{\intt {\bar{\mathcal{J}}_{\eps,\lambda}}(\Xe(s)) \d s} \nonumber\\
&\leq  \frac{1}{2} \E{\nos{x_0-Z(0)}}+\E{\intt { \bar{\mathcal{J}}_{\eps,\lambda}}(Z(s)) \d s}  \\
&+ \frac{1}{2}\E{\intt \nos{\Xe(s)-Z(s)} \d s}
+\E{\intt \ska{\Xe(s)-Z(s),G} \d s}\,,\nonumber
\end{align}
and analogically for $\eps=0$ it holds that
\begin{align}\label{SVIeps0}
\frac{1}{2}& \E{\nos{X(t)-Z(t)}}+\E{\intt { \bar{\mathcal{J}}_{\lambda}}(X(s)) \d s} \nonumber\\
&\leq  \frac{1}{2} \E{\nos{x_0-Z(0)}}+\E{\intt { \bar{\mathcal{J}}_{\lambda}}(Z(s)) \d s}  \\
&+\frac{1}{2} \E{\intt \nos{X(s)-Z(s)} \d s}
+\E{\intt \ska{X(s)-Z(s),G} \d s}\nonumber.
\end{align}
\end{defs}

The existence of SVI solutions (\ref{reg.SVI}), (\ref{SVIeps0}) follows as in \cite[Theorem 3.1]{our_paper} by the lower semicontinuity of $\bar{ \mathcal{J}}_{\eps, \lambda}$, $\bar{ \mathcal{J}}_\lambda$, respectively.
{The uniqueness of SVI solution (\ref{SVIeps0}) follows from \cite[Theorem 3.2]{Roeckner_TVF_paper}.}

To show {uniqueness of the SVI solution (\ref{reg.SVI})} we proceed as in \cite[Theorem 3.1]{our_paper}
with exception that the term $IV$ in (29) takes a different form. In particular, to obtain uniqueness we have to show the following estimate: 
\begin{align}\label{iv}
IV:=\ska{X_1^{\eps}-X^{\eps,\delta}_{2,n},-\div \fe{X^{\eps,\delta}_{2,n}}} \leq \bar{\mathcal{J}}_{\eps,0}(\Xe_1)-\bar{\mathcal{J}}_{\eps,0}(X^{\eps,\delta}_{2,n})\, .
\end{align}
{We note that term $IV$ is well defined since $\div \fe{X^{\eps,\delta}_{2,n}} \in \L$ for a.a. $(\omega, t) \in \Omega\times(0,T)$.
Indeed, from \cite[Lemma 3.2]{our_paper} for $\delta > 0$, $n< \infty$
we deduce by parabolic regularity theory that $X^{\eps,\delta}_{2,n}(\omega, t) \in \mathbb{H}^2$ for a.a. $(\omega, t) \in \Omega\times(0,T)$ and a direct calculation yields that
$$
\left|\div \fe{X^{\eps,\delta}_{2,n}}\right| \leq 2\frac{|\nabla X^{\eps,\delta}_{2,n}||\nabla^2 X^{\eps,\delta}_{2,n}|}{\big(|\nabla X^{\eps,\delta}_{2,n}|^2 + \eps^2\big)^{\frac{3}{2}}}
 + \frac{|\Delta X^{\eps,\delta}_{2,n}|}{\sqrt{|\nabla X^{\eps,\delta}_{2,n}|^2 + \eps^2}}\,.
$$
}

We show the inequality in (\ref{iv}) by the integration by parts formula using a density argument.
{We fix $(\omega, t)\in \Omega\times(0,T)$ and proceed below with 
$X_1^{\eps} \equiv X_1^{\eps}(\omega,t)$, $X^{\eps,\delta}_{2,n}\equiv X^{\eps,\delta}_{2,n}(\omega,t)$}.
We consider an approximating sequence $x_k \in C^{\infty}(\O)\cap BV(\O)$, s.t. $x_k \rightarrow X_1^{\eps}$ strongly in $\mathbb{L}^1$ and
\begin{equation}\label{xk_j}
\mathcal{J}_{\eps,{0}}(x_k) \rightarrow \mathcal{J}_{\eps,{0}}(X_1^{\eps}) \,\,\text{ for } k \rightarrow \infty\,,
\end{equation}
cf., {\cite[Theorems 10.1.2, 13.4.1 and {Remark 10.2.1}]{book_attouch}} or \cite[Theorem 5.2]{book_temam_plast}.

{Note that, since $X_1^{\eps}(\omega,t)\in \L$ it is straightforward to modify the proof of \cite[Theorems 10.1.2]{book_attouch} 
(see for instance \cite[Proposition 2.2.4]{book_attouch} for the $\mathbb{L}^p$ properties of the mollifiers)
such that the sequence $x_k$ converges strongly in $\L$:
\begin{equation}\label{l2conv}
\|x_k - X_1^{\eps}\|_{\L}\rightarrow 0 \,\,\text{ for } k \rightarrow \infty\,.
\end{equation}
}

Using the integration by parts formula \cite[Theorem 10.2.1]{book_attouch} we obtain that
\begin{align}\label{ibp}
&\ska{x_k-X^{\eps,\delta}_{2,n},-\div \fe{X^{\eps,\delta}_{2,n}}}
=\ska{\nabla(x_k-X^{\eps,\delta}_{2,n}), \fe{X^{\eps,\delta}_{2,n}}}
\nonumber
\\
 &+\int_{\partial \O} \gamma_0(x_k) \fe{X^{\eps,\delta}_{2,n}}\cdot \nu \d \mathcal{H}^{n-1}-\int_{\partial \O} \gamma_0(X^{\eps,\delta}_{2,n}) \fe{X^{\eps,\delta}_{2,n}}\cdot \nu \d \mathcal{H}^{n-1}\,,
\end{align}
where $\nu$ is the outer unit normal vector to $\partial\O$ and $\mathcal{H}^{n-1}$  is the Hausdorff measure on $\partial \O$.

Since  $X^{\eps,\delta}_{2,n}\in\Hz$ it holds that $\gamma_0(X^{\eps,\delta}_{2,n})=0$  and the second boundary integral vanishes.
The first boundary integral can be estimated as
\begin{align}\label{gamma0}
\nonumber
\int_{\partial \O} & \gamma_0(x_k) \fe{X^{\eps,\delta}_{2,n}}\cdot\nu \d \mathcal{H}^{n-1}
\leq \int_{\partial \O} \be{\gamma_0(x_k)} \be{\fe{X^{\eps,\delta}_{2,n}}\cdot\nu} \ d\mathcal{H}^{n-1}
\\
& \leq \int_{\partial \O} \be{\gamma_0(x_k)}\d \mathcal{H}^{n-1} = {\int_{\partial \O} \be{\gamma_0(X_1^{\eps})}\d \mathcal{H}^{n-1}}\,,
\end{align}
where the last equality follows from the fact that the trace of $x_k \in C^{\infty}(\O)\cap BV(\O)$ coincides with the trace of $X_1^{\eps}$, cf. \cite[Remark 10.2.1]{book_attouch}.

By the convexity of $\mathcal{J}_{\eps,0}$ we deduce that
\begin{equation}\label{convj}
\ska{\nabla(x_k-X^{\eps,\delta}_{2,n}), \fe{X^{\eps,\delta}_{2,n}}}  \leq \mathcal{J}_{\eps,0}(x_k)-\mathcal{J}_{\eps,0}(X^{\eps,\delta}_{2,n})\,.
\end{equation}


{Hence, (\ref{iv}) follows after substituting (\ref{convj}), (\ref{gamma0}) into (\ref{ibp}) and taking the limit for $k\rightarrow \infty$ and noting (\ref{xk_j})}, (\ref{l2conv}).

The rest of the proof follows analogously to the original proof of \cite[Theorem 3.1]{our_paper}.

\section{Convergence of the full discretization}\label{sec_lap}
In the proof of \cite[Lemma 4.4]{our_paper} it is concluded that 
\begin{align*}
&\frac{1}{2}\sum_{K,K'\in \mathcal{T}_h}
\bar{v}_h^TA_K^TM^{-1}A_{K'}\bar{v}_h
\left( (\bes{\nabla v_h}+\epsilon^2)_{K}^{-\frac{1}{2}}+(\bes{\nabla v_h}+\epsilon^2)_{K'}^{-\frac{1}{2}} \right)
\\
&\geq \frac{1}{2}\sum_{K,K' \in \mathcal{T}_h}\sqrt{(\bes{\nabla v_h}+\epsilon^2)_{K'}^{-\frac{1}{2}}}\bar{v}_h^TA_K^TM^{-1}A_{K'}\bar{v}_h \sqrt{(\bes{\nabla v_h}+\epsilon^2)_K^{-\frac{1}{2}}} \geq 0\,,
\end{align*}
which is not justified.
Lemma 4.4  is required to obtain the estimate (48) in \cite[Lemma 4.5]{our_paper} (note that the continuos counterpart of the estimate in Lemma 3.2 is obtained using Proposition 2.1),
which is in turn required to show \cite[Theorem 4.1]{our_paper}. 

In this section we show an analogue of the estimate in \cite[Lemma 4.4]{our_paper} for a slightly modified numerical scheme in dimension $d=1$.
Given $J \in \mathbb{N}$ and a mesh size $h=1/J$ we consider a uniform partiton $\mathcal{T}_h=\cup_{j=1}^J T_j$ of the spatial domain  $\O=(0,1)$ 
into subintervals $T_j=(x_{i-1},x_i)$ with nodes $x_j=jh$, $j=0,\dots, J$. As in \cite{our_paper} we consider a finite element space $\mathbb{V}_h\subset\mathbb{H}^1_0$
of piecewise linear globally continuous functions on subordinated to $\mathcal{T}_h$.
The standard nodal interpolation operator
$\mathcal{I}_h:C(\bar{\O})\rightarrow \mathbb{V}_h$
is defined as
\begin{align*}
\mathcal{I}_h \Phi(x_j)=\Phi(x_j) ~~ \forall j=0,\ldots, L.
\end{align*}
We define the discrete (mass-lumped) $\L$-inner product $\ska{\cdot,\cdot}_h$ on $\mathbb{V}_h$ as
\begin{align}
\ska{\phi,\psi}_h&=\into \mathcal{I}_h(\langle
\phi, \psi\rangle)(x)\d x=h \sum_{j=1}^{J-1}  \phi(x_j),\psi(x_j) ~~ \text{for}\,\, \phi,\psi \in \mathbb{V}_h, 
\label{disc.Scalar_prod}
\end{align}
with the corresponding discrete norm $\nos{\psi}_h=\ska{\psi,\psi}_h$.

It is well known that the above discrete inner product and the norm satisfy (cf. \cite{BrennerS02}): 
\begin{align}
\no{v_h}_{\mathbb{L}^2} \leq \no{v_h}_h&\leq C\no{v_h}_{\mathbb{L}^2}~~&&\forall v_h \in \mathbb{V}_h, \label{discNorm_esti}
\\
 \be{ \ska{v_h,w_h}_h-\ska{v_h,w_h}} &\leq Ch \no{v_h}_{\mathbb{L}^2}\no{w_h}_{\mathbb{H}^1}~~&&\forall v_h,w_h \in \mathbb{V}_h. \label{discScal_esti}
\end{align}
We define the mass-lumped Discrete Laplace operator $\Delta_h: \mathbb{V}_h \rightarrow \mathbb{V}_h$ through the identity
 \begin{align}\label{dis.Laplace}
 \ska{\Delta_h v_h, w_h}_h&=- \ska{\nabla v_h,\nabla w_h}.
 \end{align}

The next lemma is the counterpart of \cite[Lemma 4.4]{our_paper} for the $1d$ discrete Laplace operator (\ref{dis.Laplace}).
Numerical experiments (not stated in this paper) indicate that the result also holds for $d>1$
(possibly under some additional assumptions on the shape of the mesh). Nevertheless, the proof of the result for $d>1$ remains open, so far.
\begin{lemma}\label{lem_laplace}
Let $\Delta_h$ be the discrete Laplacian defined by \eqref{dis.Laplace}. Then for any $v_h \in \mathbb{V}_h$,
$\epsilon, h > 0$ the following inequality holds:
\begin{align*}
\ska{\fe{ v_h},\nabla(-\Delta_h v_h)}\geq 0.
\end{align*}
\end{lemma}
\begin{proof}
Since $\mathbb{V}_h$ is the space of piecewise linear functions over $\mathcal{T}_h$
,it holds for 
$v_h \in \mathbb{V}_h$ 
that
\begin{align*}
\delta_x v^j_h := \partial_x v_h(x) \big|_{T_j}=\frac{v_h(x_j)-v_h(x_j)}{h}\,.
\end{align*}
By definition (\ref{disc.Scalar_prod}) 
and (\ref{dis.Laplace}) we deduce that
\begin{align*}
\Delta_h v_h^j := \Delta_h v_h(x_j)= \frac{v_h(x_{j+1})-2v_h(x_{j})+v_h(x_{j-1})}{h^2}=\frac{\delta_x v^{j+1}_h-\delta_x v^{j}_h}{h}\,,
\end{align*}
and $\Delta_h v_h^0 = \Delta_h v_h^J = 0$.

By the above properties we deduce that
\begin{align*}
&\ska{\fe{ v_h},\nabla(-\Delta_h v_h)}=-\ska{\frac{\partial_x v_h}{\sqrt{\bes{\partial_x v_h}+\epsilon^2}},\partial_x \Delta_h v_h}
\\
 &=-\sum_{j=1}^{J}\int_{T_j}\frac{\partial_x v_h}{\sqrt{\bes{\partial_x v_h}+\epsilon^2}}\partial_x \Delta_h v_h\,\mathrm{d}x
\\
 &=-h\sum_{j=1}^{J}\frac{\delta_x v^j_h}{\sqrt{\bes{\delta_x v^j_h}+\epsilon^2}}\delta_x \Delta_h v_h^j
=-\sum_{j=1}^{J}\frac{\delta_x v^j_h}{\sqrt{\bes{\delta_x v^j_h}+\epsilon^2}} \left( \Delta_h v_h^j-\Delta_h v_h^{j-1}\right)
\\
 &=-\sum_{j=1}^{J}\frac{\delta_x v^j_h}{\sqrt{\bes{\delta_x v^j_h}+\epsilon^2}}  \Delta_h v_h^j
+\sum_{j=1}^{J}\frac{\delta_x v^j_h}{\sqrt{\bes{\delta_x v^j_h}+\epsilon^2}} \Delta_h v_h^{j-1}
\\
& =-\sum_{j=1}^{J-1}\frac{\delta_x v^j_h}{\sqrt{\bes{\delta_x v^j_h}+\epsilon^2}}  \Delta_h v_h^j+\sum_{j=1}^{J-1}\frac{\delta_x v^{j+1}_h}{\sqrt{\bes{\delta_x v^{j+1}_h}+\epsilon^2}} \Delta_h v_h^j
\\
& =\sum_{j=1}^{J-1}\left(\frac{\delta_x v^{j+1}_h}{\sqrt{\bes{\delta_x v^{j+1}_h}+\epsilon^2}}-\frac{\delta_x v^j_h}{\sqrt{\bes{\delta_x v^j_h}+\epsilon^2}}  \right)\Delta_h v_h^j
\\
& =\frac{1}{h}\sum_{j=1}^{J-1}\left(\frac{\delta_x v^{j+1}_h}{\sqrt{\bes{\delta_x v^{j+1}_h}+\epsilon^2}}-\frac{\delta_x v^j_h}{\sqrt{\bes{\delta_x v^j_h}+\epsilon^2}}  \right)(\delta_x v_h^{j+1}-{\delta_x v_h^{j}})
\\
& \geq 0
\end{align*} 
where we used the convexity of  $\sqrt{|\cdot|^2+\epsilon^2}$ to deduce the last inequality.
\end{proof}

Using the above lemma one can show the convergence for a slight modification of the fully discrete numerical scheme of \cite{our_paper} where the standard $\L$-inner product is replaced by the discrete inner product (\ref{disc.Scalar_prod}) as follows:
given $x_0,\, g \in \L$ we set $X^{0}_{\eps,h}=\mathcal{P}_h x_0$, $g^h:=\mathcal{P}_h g$ and obtain $X^{i}_{\epsilon,h}$ for $i=1,\dots, N$ as the solution of the following system:
\begin{align}\label{full_eps_TVF_L_2_data}
\ska{X^{i}_{\epsilon,h},v_h}_h = & \ska{X^{i-1}_{\epsilon,h},v}_h-\tau \ska{\fe{X^{i}_{\epsilon,h}},\nabla v_h } \\
&-\tau\lambda\ska{X^{i}_{\epsilon,h} -g^h,v_h}_h+\ska{X^{i-1}_{\epsilon,h},v_h}_h\Delta_i W &&\forall v_h \in \mathbb{V}_h \ \nonumber.
\end{align}
By the equivalence of the norms $\no{\cdot}$ and $\no{\cdot}_h$ (cf. \eqref{discNorm_esti}) the convergence of the above numerical
approximation for $d=1$ follows as in \cite{our_paper} with \cite[Lemma 4.4]{our_paper} replaced by Lemma~\ref{lem_laplace}.

We note that the convergence proof remains valid for $d\geq 1$ in the case of the time-semi discrete variant of the original numerical scheme from \cite{our_paper}:
\begin{align*}
\ska{X^{i}_{\epsilon},\varphi} = & \ska{X^{i-1}_{\epsilon},\varphi}-\tau \ska{\fe{X^{i}_{\epsilon}},\nabla \varphi } \\
&-\tau\lambda\ska{X^{i}_{\epsilon} -g,\varphi}+\ska{X^{i-1}_{\epsilon},\varphi}\Delta_i W &&\forall \varphi \in \mathbb{H}^1_0 \ \nonumber.
\end{align*}
In the semi-discrete setting one employs the continuous counterpart of Lemma~\ref{lem_laplace} and proceeds as in the proof of \cite[Lemma 3.2]{our_paper}
to obtain the space-continuous version of the stronger estimate (48) in Lemma~4.5 from \cite{our_paper}.
Then the convergence proof of the above semi-discrete numerical scheme follows analogically as in the case of the fully discrete numerical approximation;
we skip the detailed exposition for brevity and instead refer to \cite[Section 4]{stvf_new}, from where the necessary components of the proof
can be deduced.

Finally, we conclude that a convergence proof of the fully discrete numerical approximation for $d \geq 1$, which avoids the use of \cite[Lemma 4.4]{our_paper},
is provided in the upcoming paper \cite{stvf_new}.

\section*{Acknowledgement}
The authors would like to thank Martin Ondrej\'at for pointing the two mistakes in the original paper to us.

\bibliographystyle{plain}
\bibliography{refs_short}
\end{document}